\documentclass{amsart}
\usepackage[utf8]{inputenc}
\usepackage{yufei}

\title{Domination polynomial is unimodal for large graphs with a universal vertex}
\author{Shengtong Zhang}
\date{November 2021}
\address{Department of Mathematics, Massachusetts Institute of Technology, Cambridge, MA 02139}
\email{stzh1555@mit.edu}

\begin{document}
\maketitle
\begin{abstract}
    For a undirected simple graph $G$, let $d_i(G)$ be the number of $i$-element dominating vertex set of $G$. The \textbf{domination polynomial} of the graph $G$ is defined as
    $$D(G, x) = \sum_{i = 1}^n d_i(G)x^i.$$
    Alikhani and Peng conjectured that $D(G, x)$ is unimodal for any graph $G$. Answering a proposal of Beaton and Brown, we show that $D(G, x)$ is unimodal when $G$ has at least $2^{13}$ vertices and has a universal vertex, which is a vertex adjacent to any other vertex of $G$. We further determine possible locations of the mode. 
\end{abstract}
\section{Introduction}
All graphs considered are undirected and simple. Let $G = (V,E)$ be a graph on $n$ vertices. For a vertex set $S \subset V$, we define the \textbf{neighbor} $N(S)$ as the sets of vertices that are either in $S$ or is adjacent to a vertex in $S$. We say a vertex set $S$ \textbf{dominates} a vertex set $T$ if $T \subset N(S)$, and call $S$ \textbf{dominating} if $N(S) = V$. Let $d_i(G)$ be the number of $i$-element dominating vertex set of $G$. The \textbf{domination polynomial} of the graph $G$ is defined as
$$D(G, x) = \sum_{i = 1}^n d_i(G)x^i.$$
Our starting point is the following conjecture of Alikhani and Peng in \cite{AP14}. Recall a polynomial $f(x) = \sum_{i = 0}^n a_ix^i$ is \textbf{unimodal} if there exists a $k \in [0, n]$ such that $a_i \leq a_{i + 1}$ for any $i < k$, and $a_i \geq a_{i + 1}$ for any $i \geq k$. We call the largest such $k$ the \textbf{mode} of $f(x)$.
\begin{conjecture}
\label{conj:main}
For any graph $G$, the domination polynomial $D(G, x)$ is unimodal.
\end{conjecture}
This conjecture seems to be very difficult in general. See \cite{B17} for a survey of recent progress toward this conjecture. The main motivation of our study is a proposal  Beaton and Brown made in \cite{BB20}. They considered graphs with \textbf{universal vertices}, where a vertex $v$ is called universal if $N(\{v\}) = V$. Beaton and Brown computed that all graphs with at most $10$ vertices and at least one universal vertex have unimodal domination polynomial, with mode either $\floor{\frac{n}{2}}$ or $\floor{\frac{n + 1}{2}}$. The proposal was considered by Burcroff and O'Brien in \cite{BB21}, where they proved that a graph $G$ on $n$ vertices and $m$ universal vertices satisfy $d_i(G) \geq d_{i + 1}(G)$ for any
$$i \geq \left(\frac{1}{2} + \frac{1}{2^{m + 1}}\right)n.$$
In this paper we provide an affirmative answer to Beaton and Brown's proposal, showing both unimodality of the domination polynomial and where the mode is, provided that the number of vertices is sufficiently large.
\begin{theorem}
\label{thm:main}
For $n \geq 2^{13}$, if $G$ is a graph on $n$ vertices with at least one universal vertex, then $D(G, x)$ is unimodal. Furthermore, the mode of $D(G,x)$ lies in the interval $[n/2, n/2 + \log_2 n + 2]$. 
\end{theorem}
We also provide a construction showing the bound on the mode is tight up to a constant factor.
\begin{theorem}
\label{thm:cons}
For $n \geq 2^{20}$, there exists a graph $G$ on $n$ vertices with at least one universal vertex such that the mode of $D(G, x)$ at least $n/2 + \frac{1}{999}\log_2 n$. 
\end{theorem}
The main tools we use are two novel theorems about the mode and concavity of the domination polynomial. In \cite[Lemma 3.1]{BB20}, Beaton and Brown defined the quantity $r_k(G) = d_k(G) / \binom{n}{k}$, and showed that if $r_k(G) \geq (n - k) / (k + 1)$ then the mode of $D(G, x)$ is at most $k$. Our theorems strengthen this observation even when there is no universal vertex.
\begin{theorem}
\label{thm:mode}
Let $G$ be any graph on $n$ vertices with $h$ universal vertices. For any $k \in [n/2, n]$, if there exists a positive integer $\alpha$ such that
$$d_k(G) > \frac{1}{2k + 1 - n} \cdot \left(n^2 \cdot \binom{n - \alpha - 1}{k} + \alpha \cdot \binom{n - h}{k}\right)$$
then the sequence $d_k(G), d_{k + 1}(G), \cdots, $ is decreasing, so the mode of the domination polynomial of $G$ is at most $k$.
\end{theorem}
\begin{theorem}
\label{thm:unimodality}
Let $G$ be any graph on $n$ vertices with $h$ universal vertices. For any $k \in \left[n/2, n/2 + \sqrt{n} / 4\right]$, if there exists a positive integer $\alpha$ such that
$$d_k(G) > \frac{1}{k + 1} \cdot \left(n^3 \cdot \binom{n - \alpha - 1}{k} + 2\alpha^2 \cdot \binom{n - h}{k}\right)$$
then we have $2d_{\ell + 1}(G) > d_{\ell}(G) + d_{\ell + 2}(G)$ for any $\ell \in \left[k, n/2 + \sqrt{n} / 4\right]$.
\end{theorem}
The paper is organized as follows. In \cref{sec:mode-concave} we prove our mode and concavity theorems, and in \cref{sec:universal} we establish our theorems about graphs with a universal vertex.
\section{Mode and Concavity Theorems}
\label{sec:mode-concave}
Throughout this section, we fix a graph $G = (V, E)$ on $n$ vertices. For convenience, we abuse notation and let $N(S)$ denote the size of the neighbor of $S$, and let $N(v) = N(\{v\})$ for each vertex $v$. For vertex set $S \subset V$, let $D(S)$ be the number of vertices $v$ such that $\{v\}$ dominates $S$, and let $E_k^S$ be the family of $k$-element subsets of $V$ that dominates the vertex set $V - S$ but does not dominate any vertex in $S$. We also let $E_k^S$ denote the cardinality of this family. For example, $E_k^{\emptyset} = d_k(G)$ while $E_k^V = 0$ for any $k > 0$.

We first recall a handy lemma derived by Beaton and Brown.
\begin{lemma}[\cite{BB20}, Section 3] For any $k \in [0,n - 1]$, we have
\label{lem:BB}
$$\frac{d_{k + 1}(G)}{\binom{n}{k + 1}} \geq \frac{d_{k}(G)}{\binom{n}{k}}.$$
\end{lemma}
Our first new observation is the following lemma. The lemma is motivated by arguments in \cite{BB19}, though they count from the ``$k + 1$-side" while we count from the ``$k$-side".
\begin{lemma}
\label{lem:eq1}
We have the identity
$$(k + 1)(d_{k + 1}(G) - d_k(G)) = \sum_{T \subset V, T\neq \emptyset} E_k^T D(T) - (2k + 1 - n)d_k(G).$$
\end{lemma}
\begin{proof}
We apply double counting to the set 
$$H = \{(U, W): U \subset W, \abs{U} = k, \abs{W} = k + 1, W \text{ dominates }V\}.$$ On one hand, each dominating set $W$ of size $k + 1$ corresponds to exactly $k + 1$ choices of $(U,W) \in H$, so we have $\abs{H} = (k + 1)d_{k + 1}(G)$. On the other hand, each $k$-element dominating $U$ corresponds to $n - k$ choices of $(U, W) \in H$ by adding any vertex not in $U$ to $U$. For $S \neq \emptyset$ and any $U \in E_k^S$, $W = U \cup \{v\} \in d_{k + 1}(G)$ if and only if $v$ dominate $S$ and $v \notin U$. As $U$ cannot contain any vertex that dominate $S$ by definition, we conclude that each $U \in E_k^S$ corresponds to exactly $D(S)$ pairs of $(U, W)\in H$. Thus
$$(k + 1)d_{k + 1}(G) = \abs{H} = (n - k)d_k(G) + \sum_{T \subset V, T\neq \emptyset} E_k^T D(T).$$
Rearranging gives the desired identity.
\end{proof}
We now show \cref{thm:mode}.
\begin{proof}[Proof of \cref{thm:mode}]
First we show that $d_k(G) \geq d_{k + 1}(G)$. It suffices to upper bound
$$\sum_{T \subset V, T\neq \emptyset} E_k^T D(T).$$
We split all possible $T$ into two subsets. Let $B$ be the set of vertices $v\in V$ with $N(v) > \alpha$. For $v \in B$, let $\cT_v$ be the family of nonempty vertex sets $T$ that contains $v$. Let $\cT_0$ be the family of nonempty vertex sets $T$ that does not contain any vertex in $B$. Then
$$\sum_{T \subset V, T\neq \emptyset} E_k^T D(T) \leq \sum_{v \in B} \sum_{T \in \cT_v} E_k^T D(T) + \sum_{T \in \cT_0} E_k^T D(T).$$
For each $v \in B$, we have
$$\sum_{T \in \cT_v} E_k^T D(T) \leq n \sum_{T \in \cT_v} E_k^T.$$
The sum
$$\sum_{T \in \cT_v} E_k^T$$
counts the number of $k$-element vertex sets $S$ that does not dominate $v$, so 
$$\sum_{T \in \cT_v} E_k^T = \binom{n - N(v)}{k} \leq \binom{n - \alpha - 1}{k}.$$
Therefore we have
$$\sum_{v \in B} \sum_{T \in \cT_v} E_k^T D(T) \leq n^2 \cdot\binom{n - \alpha - 1}{k}.$$
On the other hand, for each $T \in \cT_0$ we have $D(T) \leq \alpha$, and by definition any set not dominating $T$ must not contain any universal vertex. Therefore
$$\sum_{T \in \cT_0} E_k^T D(T) \leq \alpha\sum_{T \in \cT_0} E_k^T \leq \alpha\cdot \binom{n - h}{k}.$$
We sum the two estimates to get
$$ \sum_{T \subset V, T\neq \emptyset} E_k^T D(T) \leq n^2 \cdot \binom{n - \alpha - 1}{k} + \alpha \cdot \binom{n - h}{k} < (2k + 1 - n) d_k(G).$$
Thus \cref{lem:eq1} implies $d_{k}(G) \leq d_{k + 1}(G).$

For any $\ell \geq k$, we show $d_{\ell}(G) \geq d_{\ell + 1}(G)$. It suffices to show that
$$d_\ell(G) > \frac{1}{2\ell + 1 - n} \cdot \left(n^2 \cdot \binom{n - \alpha - 1}{\ell} + \alpha \cdot \binom{n - h}{\ell}\right)$$
which is equivalent to
$$\frac{d_{\ell}(G)}{\binom{n}{\ell}} > \frac{1}{2\ell + 1 - n} \cdot \left(n^2 \cdot \frac{\binom{n - \alpha - 1}{\ell}}{\binom{n}{\ell}} + \alpha \cdot \frac{\binom{n - h}{\ell}}{\binom{n}{\ell}}\right).$$
The left hand side is non-decreasing in $\ell$ by \cref{lem:BB}, while the right hand side is decreasing in $\ell$. As the inequality holds for $\ell = k$, it holds for any $\ell \geq k$. 
\end{proof}
Now we turn to the concavity theorem, which is proved using another double counting argument. For vertex sets $T \subset S\subset V$, let $D(S:T)$ be the number of vertices that dominates $S - T$ but does not dominate any vertex in $T$.
\begin{lemma}
For any nonempty vertex set $T$, we have
$$(k + 1)(E_{k + 1}^T - E_k^T) = \sum_{S \subset V, S \supsetneq T} E_k^S D(S:T) - (2k + 1 + N(T) - n)E_k^T.$$
\end{lemma}
\begin{proof}
We apply double counting to the set $H = \{(U, W): U \subset W, \abs{U} = k, W \in E_{k + 1}^T\}$. On one hand, each $W \in E_{k + 1}^T$ corresponds to exactly $k + 1$ choices of $(U,W) \in H$, so we have $\abs{H} = (k + 1)d_{k + 1}(G)$. On the other hand, each $U \in E_{k}^T$ corresponds to $n - k - N(T)$ choices of $(U, W) \in H$ by adding any vertex not in $U \cup N(T)$ to $U$, where the union is disjoint as $U \in E_k^T$. For any $S \supsetneq T$ and any $U \in E_k^S$, $W = U \cup \{v\} \in E_k^T$ if and only if $v$ dominates $S - T$ and $v \notin U \cup N(T)$. As $U$ cannot contain any vertex in the neighbor of $S$ by definition, we conclude that each $U \in E_k^S$ corresponds to exactly $D(S:T)$ pairs of $(U, W)\in H$. Thus
$$(k + 1)E_{k + 1}^T = \abs{H} = (n - k - N(T))E_k^T + \sum_{S \subset V, S \supsetneq T} E_k^S D(S:T).$$
Rearranging gives the desired identity.
\end{proof}
The lemma implies the following inequality.
\begin{lemma}
\label{lem:eq2}
For any $k \in \left[n/2, n/2 + \sqrt{n} / 4\right]$ we have
$$2d_{k + 1}(G) - d_{k}(G) - d_{k + 2}(G) \geq  \frac{1}{k + 1} d_k(G) - \frac{1}{(k + 1)^2}\sum_{S \subset V, S \neq \emptyset} E_k^S D'(S)$$
where $D'(S)$ is the cardinality of the set
$$\{(u_1, u_2): u_1, u_2 \in V, \{u_1\}, \{u_2\}\text{ does not dominate }S\text{ but }\{u_1, u_2\}\text{ does}\}.$$
\end{lemma}
\begin{proof}
By \cref{lem:eq1} we have the identities
$$d_{k + 1}(G) - d_k(G) = \frac{1}{k + 1} \left(\sum_{T \subset V, T\neq \emptyset} E_k^T D(T) - (2k + 1 - n)d_k(G)\right),$$
$$d_{k + 2}(G) - d_{k + 1}(G) = \frac{1}{k + 2} \left(\sum_{T \subset V, T\neq \emptyset} E_{k + 1}^T D(T) - (2k + 3 - n)d_{k + 1}(G)\right).$$
Therefore we have
\begin{align}
&2d_{k+1}(G) - d_{k + 2}(G) - d_{k}(G) \notag\\
    =&\left(\frac{1}{k + 1} \left(\sum_{T \subset V, T\neq \emptyset} E_k^T D(T) \right) - \frac{1}{k + 2} \left(\sum_{T \subset V, T\neq \emptyset} E_{k + 1}^T D(T)\right)\right) \notag\\
    +& \left(\frac{2k + 3 - n}{k + 2}d_{k + 1}(G) - \frac{2k + 1 - n}{k + 1} d_k(G)\right). \label{eq:lem2eq1}
\end{align}
By \cref{lem:BB}, we have
$$d_{k + 1}(G) \geq \frac{n - k}{k + 1}d_k(G),$$
and since $k \in \left[n/2, n/2 + \sqrt{n} / 4\right]$, we also have
$$\frac{2k + 3 - n}{k + 2} \geq \frac{2k + 2 - n}{n - k}.$$
Therefore we conclude
$$\frac{2k + 3 - n}{k + 2}d_{k + 1}(G) - \frac{2k + 1 - n}{k + 1} d_k(G) \geq \frac{1}{k + 1}d_k(G).$$
It remains to estimate the first term in \eqref{eq:lem2eq1}. We first note that
\begin{align*}
 & \frac{1}{k + 1} \left(\sum_{T \subset V, T\neq \emptyset} E_k^T D(T) \right) - \frac{1}{k + 2} \left(\sum_{T \subset V, T\neq \emptyset} E_{k + 1}^T D(T)\right) \\
 \geq & \frac{1}{k + 1} \sum_{T \subset V, T\neq \emptyset} (E_k^T - E_{k + 1}^T) D(T).
\end{align*}
Now we apply \cref{lem:eq2}, and obtain
\begin{align*}
    &\sum_{T \subset V, T\neq \emptyset} (E_k^T - E_{k + 1}^T) D(T) \\
    =& \frac{1}{k + 1} \sum_{T \subset V, T\neq \emptyset} \left( (2k + 1 + N(T) - n)E_k^T - \sum_{S \subset V, S \supsetneq T} E_k^S D(S:T)\right) D(T) \\
    =& \frac{1}{k + 1} \sum_{S \subset V, S\neq \emptyset} \left((2k + 1 + N(S) - n)D(S) - \sum_{T \subsetneq S, T \neq \emptyset} D(S:T)D(T)\right) E_k^S.
\end{align*}
Now fix any $S \subset V$. By definition of $D(S:T)$ and $D(T)$, the sum
$$\sum_{T \subsetneq S, T \neq \emptyset} D(S:T)D(T)$$
is equal to the cardinality of the set
$$H_1 = \{(T, u_1, u_2): T \subsetneq S, T \neq \emptyset, \{u_1\} \text{ dominates }S - T\text{ but is not in the neighbor of }T, \{u_2\}\text{ dominates }T\}.$$
We note that $T = S\backslash N(u_1)$. Therefore the cardinality of $H_1$ is the same as the cardinality of
$$H_2 = \{(u_1, u_2): u_1 \text{ is in the neighbor of }S\text{ but does not dominate }S, \{u_1, u_2\}\text{ dominates }S\}.$$
We can partition $H_2$ based on whether $u_2$ dominates $S$. If $u_2$ dominates $S$, then $u_1$ can be any vertex in the neighbor of $S$ but not dominating $S$, so the number of such elements in $H_2$ is $(N(S) - D(S))D(S)$. On the other hand, the number of elements with $u_2$ not dominating $S$ is $D'(S)$ be definition. Therefore we get
$$\sum_{T \subsetneq S, T \neq \emptyset} D(S:T)D(T) = D(S)N(S) - D(S)^2 + D'(S).$$
So we conclude that
\begin{align*}
    &\sum_{T \subset V, T\neq \emptyset} (E_k^T - E_{k + 1}^T) D(T) \\
    =& \frac{1}{k + 1} \sum_{S \subset V, S\neq \emptyset} \left((2k + 1 - n)D(S) + D(S)^2 - D'(S)\right) E_k^S \geq -\frac{1}{k+1}\sum_{S \subset V, S\neq \emptyset} D'(S)E_k^S.
\end{align*}
Combining the estimates on the two term of \eqref{eq:lem2eq1} gives the desired result.
\end{proof}
We are now ready to show our concavity theorem.
\begin{proof}[Proof of \cref{thm:unimodality}]
The proof is the same as the proof of \cref{thm:mode}. We use the same definition of $B$, $\cT_v$ and $\cT_0$ as in the proof of \cref{thm:mode}. Then
$$\sum_{S \in \cT_v} E_k^S D'(S) \leq n^2 \sum_{S \in \cT_v} E_k^S \leq n^2 \cdot \binom{n - \alpha - 1}{k}.$$
For any $S \subset \cT_0$, recall that $D'(S)$ is the cardinality of the set
$$\{(u_1, u_2): u_1, u_2 \in V, \{u_1\}, \{u_2\}\text{ does not dominate }S\text{ but }\{u_1, u_2\}\text{ does}\}.$$
Fix any $s_1 \in S$. Either $u_1$ or $u_2$ must dominate $s_1$. If $u_1$ dominate $s_1$, then since $N(s_1) \leq \alpha$, there are at most $\alpha$ choices of $u_1$. Fixing $u_1$, there must be some $s_2 \in S$ not dominated by $u_1$. So $u_2$ must be chosen to dominate $s_2$, so there are at most $\alpha$ choices of $u_2$ given a choice for $u_1$. Thus there are at most $\alpha^2$ choices of $(u_1, u_2)$ given that $u_1$ dominate $s_1$. Symmetrically, there are at most $\alpha^2$ choices of $(u_1, u_2)$ given that $u_2$ dominate $s_1$. So we conclude that $D'(S) \leq 2\alpha^2$. Therefore
$$\sum_{S \in \cT_0} E_k^S D'(S) \leq 2\alpha^2\sum_{S \in \cT_0} E_k^S \leq 2\alpha^2 \cdot \binom{n - h}{k}.$$
Combining the two estimates, we conclude that 
$$\sum_{S \subset V, S \neq \emptyset} E_k^S D'(S) \leq \sum_{v \in V} \sum_{S \subset \cT_v}  E_k^S D'(S) + \sum_{S \subset \cT_0}  E_k^S D'(S) \leq n^3 \cdot \binom{n - \alpha - 1}{k} + 2\alpha^2 \cdot \binom{n - h}{k}.$$
By \cref{lem:eq2}, we conclude that $2d_{k + 1}(G) > d_{k}(G) + d_{k + 2}(G)$. Analogous to the proof of \cref{thm:mode}, for any $\ell \geq k$ we have
$$d_\ell(G) > \frac{1}{\ell + 1} \cdot \left(n^3 \cdot \binom{n - \alpha - 1}{\ell} + 2\alpha^2 \cdot \binom{n - h}{\ell}\right)$$
so we have $2d_{\ell + 1}(G) > d_{\ell}(G) + d_{\ell + 2}(G)$ for any $\ell \in \left[k, n/2 + \sqrt{n} / 4\right]$.
\end{proof}
\section{graphs with a universal vertex}
\label{sec:universal}
We are now ready to establish our result on graphs with a universal vertex. 
\subsection{Proof of \cref{thm:main}}
Let $G = (V,E)$ be any graph on $n \geq 2^{13}$ vertices with a universal vertex $v$. Then any vertex set containing $v$ is a dominating set of $G$, so we have
$$d_k(G) \geq \binom{n - 1}{k - 1}.$$
In \cite[Proposition 2.4]{BB20}, Beaton and Brown showed that $d_i(G) \leq d_{i + 1}(G)$ for any $i < n / 2$, so the mode of $D(G, x)$ is at least $n / 2$. In the other direction, we apply \cref{thm:mode} with $k = \ceil{n/2 + \log_2 n} + 1$ and $\alpha = \floor{2\log_2 n}$. Then we have
$$n^2 \cdot \binom{n - \alpha - 1}{k} \leq n^2 \cdot 2^{-\alpha - 1}\binom{n}{k} < \binom{n}{k} < 2\binom{n - 1}{k - 1}.$$
and
$$\alpha \cdot \binom{n - h}{k} < (2k  - 1 - n) \binom{n - 1}{k} < (2k - 1 - n)\binom{n - 1}{k - 1}.$$
Thus we have
$$\frac{1}{2k + 1 - n} \cdot \left(n^2 \cdot \binom{n - \alpha - 1}{k} + \alpha \cdot \binom{n - h}{k}\right) < \binom{n - 1}{k - 1} \leq d_k(G).$$
So \cref{thm:mode} implies $d_\ell(G) > d_{\ell + 1}(G)$ for any $\ell \geq \ceil{n/2 + \log_2 n} + 1$. Thus the mode of $D(G, x)$ is in the interval $[n/2, n/2 + \log_2 n + 2]$.

To finish the proof, we establish the unimodality of $D(G, x)$. It suffices to show that for any $\ell \in [n/2, n/2 + \log_2 n + 2]$ we have $2d_{\ell + 1}(G) > d_\ell(G) + d_{\ell + 2}(G)$. We use \cref{thm:unimodality} with $k = \ceil{n/2}$ and $\alpha = \floor{3\log_2 n}$. Then we have
$$\binom{n - \alpha - 1}{k} < 2^{-\alpha-1} \binom{n}{k} < n^{-3}\binom{n}{k} < 2n^{-3}\binom{n - 1}{k - 1}$$
and
$$2\alpha^2 \cdot \binom{n - h}{k} \leq 2\alpha^2  \cdot \binom{n - 1}{k - 1} \leq 18\log_2^2 n \binom{n - 1}{k - 1} < (k - 1)\binom{n - 1}{k - 1}.$$
So we conclude that
$$\frac{1}{k + 1} \cdot \left(n^3 \cdot \binom{n - \alpha - 1}{k} + 2\alpha^2 \cdot \binom{n - h}{k}\right) < \binom{n - 1}{k - 1} \leq d_{k}(G).$$
By \cref{thm:unimodality}, for any $\ell \in [n/2, n/2 + \sqrt{n} / 4]$ we have $2d_{\ell + 1}(G) > d_\ell(G) + d_{\ell + 2}(G)$. As $\sqrt{n} / 4 > \log_2 n + 2$, we conclude that $D(G, x)$ is unimodal.

\subsection{Proof of \cref{thm:cons}} We take any regular graph $G_0 = (V_0, E_0)$ on $n - 1$ vertices with degree $d = 2\floor{\log_2 (4n) / 2} + 2$ and girth at least $5$, then add a universal vertex $u$ to form the graph $G$. By \cref{lem:eq2}, it suffices to show that for any $k \in [n/2, n/2 + \frac{1}{999}\log_2 n]$ we have
$$\sum_{T \subset V_0, T\neq \emptyset} E_k^T D(T) - (2k + 1 - n)d_k(G) > 0.$$
On one hand, we have the trivial bound
$$d_k (G) \leq \binom{n}{k}.$$
On the other hand, note that
$$\sum_{T \subset V, T\neq \emptyset} E_k^T D(T) \geq \sum_{v \in V_0} E_k^{\{v\}} D(\{v\}).$$
For any $v \in V_0$, a $k$-element vertex subset of $V_0$ does not dominant $v$ if and only if it does not contain any vertex in $N(v)$, so the number of such sets is
$$\binom{n - d - 1}{k}.$$
On the other hand, if a $k$-element vertex subset of $V_0$ does not dominate $v$ and is not in $E_k^T$, then it must not dominate another vertex $w \in V_0$, so does not contain any vertex in $N(\{v, w\})$. As $G$ has girth at least $5$, we have $N(\{v, w\}) \geq 2d$. So the number of such sets is at most
$$\sum_{w \in V_0, w\neq v}\binom{ n - N(\{v, w\})}{k} \leq (n-2)\binom{n - 2d}{k}.$$
Thus we conclude that
$$\abs{E_k^{\{v\}}} \geq \binom{n - d - 1}{k} - (n-2)\binom{n - 2d}{k} \geq (1 - (n - 2)2^{-d + 1})\binom{n - d - 1}{k} \geq \frac{1}{2}\binom{n - d - 1}{k}.$$
Therefore
$$\sum_{v \in V_0} E_k^{\{v\}} D(\{v\}) \geq \frac{n - 1}{2}\cdot \binom{n - d - 1}{k}\cdot (d + 1).$$
Now we note that
$$\binom{n - d - 1}{k} = \binom{n}{k} \cdot \frac{(n - k - 1)(n - k - 2) \cdots (n - k - d - 1)}{n \cdot (n - 1) \cdots (n - d)} \geq \binom{n}{k} \cdot \left(\frac{n - k - d - 1}{n - d}\right)^{d + 1}.$$
We note that
$$\frac{n - k - d - 1}{n - d} \geq \frac{1}{2} - \frac{d}{n}.$$
So we have
$$\binom{n - d - 1}{k} \geq \binom{n}{k}  \cdot \frac{1}{2^{d + 2}} \geq \binom{n}{k} \cdot \frac{1}{64n}.$$
Thus we obtain the desired inequality
$$\sum_{v \in V_0} E_k^{\{v\}} D(\{v\}) \geq \frac{n - 1}{2}\cdot (d + 1) \cdot \frac{1}{64n}\binom{n}{k} > (2k + 1 - n)\binom{n}{k} \geq (2k + 1 - n) d_k(G).$$
We conclude that the graph $G$ satisfies the conditions of \cref{thm:cons}.
\section{Concluding remarks}
The technique developed in \cref{sec:mode-concave} is useful for showing unimodality of the domination polynomial for graphs with high degree vertices. However, it sheds no light to the case when the maximum degree of $G$ is small, for example when $G$ is a cubic graph or a tree. New techniques might be needed to deal with such $G$. A possible direction was suggested in \cite{BB21}, where the authors proposed the following conjecture.
\begin{conjecture}
For any tree T and a leaf $v \in V(T)$, if both $D(T)$ and $D(T\backslash v)$ are
unimodal, then their modes are of distance at most 1.
\end{conjecture}
It might also be interesting to get rid of the large size constraint $n \geq 2^{13}$ on the graph $G$.
\section*{Acknowledgements}
The author's research is self-funded. The author thanks I. Beaton for communicating the problem at CanaDAM 2021. 

\end{document}